\documentclass[oneside,english,reqno]{amsart}
\usepackage{algorithm}
\usepackage{algorithmic}
\usepackage{multirow}
\usepackage{subfigure}
\usepackage{enumerate}

\usepackage[T1]{fontenc}
\usepackage[latin9]{inputenc}
\usepackage{babel}

\usepackage{amsthm}
\usepackage{amssymb}
\usepackage{graphicx}
\usepackage[unicode=true,pdfusetitle,
bookmarks=true,bookmarksnumbered=false,bookmarksopen=false,
breaklinks=false,pdfborder={0 0 1},backref=false,
hidelinks]
{hyperref}

\makeatletter
\numberwithin{equation}{section}
\numberwithin{figure}{section}

\makeatother
\newtheorem{theorem}{Theorem}[section]

\newtheorem{definition}[theorem]{Definition}
\newtheorem{example}[theorem]{Example}

\newtheorem{lemma}{Lemma}[section]

\newtheorem{proposition}[theorem]{Proposition}
\newtheorem{remark}[theorem]{Remark}

\newcommand{\gph}{\text{gph}\,}
\newcommand{\hilbertH}{ \mathcal{H}}
\newcommand{\hilbertG}{ \mathcal{G}}
\newcommand{\dist}{\text{dist}}
\newcommand{\rank}{\text{rank}\,}

\newcommand{\setD}{\mathcal{D}}
\usepackage{xcolor}
\begin{document}

%\title{Elsevier \LaTeX\ template\tnoteref{mytitlenote}}
\title[On Lipschitz-like continuity]{On Lipschitz-like continuity of a class of set-valued mappings
%	\tnoteref{mytitlenote}
}


%\author{Ewa M. Bednarczuk\fnref{myfootnote1}, Krzysztof W. LeÅniewski\fnref{myfootnote2}, Krzysztof E. Rutkowski\fnref{myfootnote3}}
%\fntext[myfootnot1e]{Instytut BS, PW}
%\fntext[myfootnote2]{IBS, PW}
%\fntext[myfootnote3]{Politechnika}
%\textbf{}

%\author[mymainaddress,mysecondaryaddress]{Ewa M. Bednarczuk}
%\author[mymainaddress,mysecondaryaddress]{Ewa M. Bednarczuk}
%%\ead[url]{www.elsevier.com}
%
%\author[mymainaddress,mysecondaryaddress]{Krzysztof W. Le\'sniewski}
%
%
%
%\author[mymainaddress,mysecondaryaddress]{Krzysztof E. Rutkowski%\corref{mycorrespondingauthor}
%}
%%\cortext[mycorrespondingauthor]{Corresponding author}
%%\ead{k.rutkowski@mini.pw.edu.pl}
%%\ead[url]{www.elsevier.com}
%
%\address[mymainaddress]{Warsaw University of Technology, 00-662 Warszawa, Koszykowa 75}
%\address[mysecondaryaddress]{Systems Research Institute, PAS, 01-447 Warsaw, Newelska 6}
\author{
	Ewa M. Bednarczuk$^1$
}
\author{
	Leonid I. Minchenko$^2$
}
\author{
	Krzysztof E. Rutkowski$^3$ 
}
\thanks{$^1$ Systems Research Institute of the Polish Academy of Sciences, Warsaw University of Technology%\href{mailto:e.bednarczuk@mini.pw.edu.pl}{e.bednarczuk@mini.pw.edu.pl} .
}
\thanks{$^2$ Belarus State University of Informatics and Radioelectronics, Minsk, Belarus	 %\href{mailto:k.lesniewski@mini.pw.edu.pl}{k.lesniewski@mini.pw.edu.pl} .
	}
\thanks{$^3$ Warsaw University of Technology, Systems Research Institute of the Polish Academy of Sciences 		 %\href{mailto:k.rutkowski@mini.pw.edu.pl}{k.rutkowski@mini.pw.edu.pl} .
	}

\begin{abstract}
We study set-valued mappings defined by solution sets of parametric systems of equalities and inequalities. We prove Lipschitz-like continuity of these mappings under relaxed constant rank constraint qualification. 
\end{abstract}
\maketitle

%\begin{keyword}  tangent cone, relaxed constant rank condition, Hilbert space, rank theorem, Ljusternik theorem, Lagrange multipliers, Abadie condition 
%%\texttt{elsarticle.cls}\sep \LaTeX\sep Elsevier \sep template
%\MSC[2010] 00-01\sep  99-00
%\end{keyword}

%\linenumbers

%\section{Introduction}

%{\color{red}Conditions ensuring the possibility that the tangent cone
% to standard feasible sets appearing in nonlinear constrained optimization can be represented via the linearized cone lay at the core of 
% optimality conditions.... For equality constraints.... for inequality constraints... \cite{necessary_conditions_Avakov}
% 
% In the present paper we discuss}

\section{ Introduction}

Properties of set-valued mappings given by  systems of equalities and inequalities play a significant role in parametric optimization. In particular, considerable effort is directed towards formulating  conditions ensuring Lipschitz-type continuities of these mappings, namely their calmness and the pseudo-Lipschitz continuity (also referred to as Lipschitz-like continuity or the Aubin property) \cite{MR736641,MR3022301,MR751246,MR3341662,MR2801389}.
%[1,5,7,8,13]. 

The present paper is devoted to  sufficient conditions for R-regularity (Definition \ref{definition_R_regularity}) and pseudo-Lipschitz continuity for set-valued mappings defined by solution sets of parametric constrained systems.  R-regularity is a variant of a much more general property, called metric regularity, intensively studied in \cite{MR2445245,MR2548956,MR3288139,MR1777352,MR3727108,MR3575646}.

Let $\hilbertH$ be a Hilbert space and $\hilbertG$ be a normed space.
Let us consider a parametric nonlinear programming problem:
\begin{equation}\label{eq:1.1}
\begin{array}{ll}
\text{minimize}& f(p,x)\\
\text{subject to} & x \in F(p) = \{ x \in \hilbertH \mid {h_i}(p,x) \le 0,\ i \in I,\ h_i(p,x) = 0,\ i \in I_0\}, 
\end{array}
\end{equation}
where $p \in \hilbertG$ is a parameter, $x\in \hilbertH$ stands for the decision variable, $I = \{ 1,\dots,m\} $, ${I_0} = \{ m + 1,\dots,n\} $ (we admit the case $I_0=\emptyset$). Functions $f,{h_i}:\ \hilbertG\times \hilbertH\rightarrow\mathbb{R}$, $i=1,\dots, n$, are assumed to be (jointly) continuous together with their partial gradients with respect to $x$,  ${\nabla _x}f$ and ${\nabla _x}{h_i}$, $ i=1,\dots, n$.

In the present paper we prove Lipschitz-likeness of the set-valued mapping $F$ defined in \eqref{eq:1.1}. We generalize results from \cite{MR2783218} and \cite{2018arXiv181105166B}. In \cite{MR2783218} the respective results are obtained under stronger assumptions of functions $h_i$, $i\in I_0\cup I$, while in \cite{2018arXiv181105166B} the Lipschitz-likeness of $F$ is obtained for $h_i(p,x)=\langle x \mid g_i(p)\rangle -f_i(p)$, $i\in I_0\cup I$, where $f_i:\ \hilbertG \rightarrow \mathbb{R}$, $g_i:\ \hilbertG  \rightarrow \hilbertH$, $i\in I_0\cup I$, are locally Lipschitz functions.
We also correct the mistake in the proof of Lemma 3 of \cite{MR2783218}.
%	note has two goals. Firstly, its goal is to correct a mistake in the proof of one from results in \cite{MR2783218}. The second objective is to generalize results \cite{MR2783218}. Unlike \cite{MR2783218} the results of the present note hold for mappings which are RCRCQ, R-regular and pseudo-Lipschitz  continuous relative to some sets and proved without additional assumptions \cite{MR2783218} about Lipschitz continuity of  gradients ${\nabla _x}f$ and ${\nabla _x}{h_i}$. cos  tu jest niejasne!!}

For set-valued mapping $F:\ \hilbertG \rightrightarrows \hilbertH$ defined in \eqref{eq:1.1}, its domain and graph are defined by  $domF = \{ p \in \hilbertG \mid F(p) \ne \emptyset \} $ and $grF = \{ (p,x)|x \in F(p),\ p \in \hilbertG\} $, respectively.

The tangent cone $T(F(p),x)$ and the linearized  cone $\Gamma(F(p),x)$ to $F(p)$ at $x \in F(p)$ are defined, respectively, as follows
\begin{align*}
&T(F(p),x) = \{  d \in \hilbertH \mid \exists {t_k} \downarrow 0,{ d^k} \to d\text{ such that }x + {t_k}d^k \in F(p),\ k = 1,2,\dots\} ,\\
&\Gamma (F(p),x) = \{ d \in \hilbertH\mid\langle {\nabla _x}{h_i}(p,x),d\rangle  \le 0:i \in I(p,x),\ \langle {\nabla _x}{h_i}(p,x),d\rangle  = 0:i \in {I_0}\}.
\end{align*}

Organization of the paper is as follows. In section 2 we provide basing concepts. Section 3 is devoted to the concept of $R-$regularity. In section 4 we investigate relationships between relaxed constant rank condition and the R-regularity of $F$. In section 5 we prove Lipschitz-likeness of $F$ under relaxed constant rank condition. In section 6 some applications to bilevel programming are discussed.
\section{Basic concepts and definitions}
This section contains some background material (see, i.e., \cite{MR736641,MR2179245,MR3341662,MR1946832,MR2801389,MR1491362}) %[1,5,7,8,10,13])
which will be used in the sequel. 
% We give only definitions and some concise technical results that will be needed in the note. For more detailed information on the subject, see the works in the references.

We denote ${V_\delta }({p^0}) = {p^0} + \delta B$, ${V_\varepsilon }({x^0}) = {x^0} + \varepsilon B$,  where $B$ is the open unit ball centered at $0$ in the respective space, $\dist(v,C):=\inf\{ \|v-c\|,\ c\in C\}$ is the distance between  point $v$ and  set $C$,  where $\| v \|$  is the norm of  vector $v$.  

\begin{definition} A set-valued mapping $F$ is lower semicontinuous (lsc) at  $({p^0},{x^0}) \in grF$  (relative to $P \subset \hilbertG$) if  for any neighbourhood $V({x^0})$ there is a neighborhood $V({p^0})$ such that $F(p) \cap V({x^0}) \ne \emptyset $ for all $p \in V({p^0})$ (for all $p \in V({p^0}) \cap P$).
\end{definition}

\begin{definition} A set-valued mapping $F$ is  lower Lipschitz  continuous at  $({p^0},{x^0}) \in grF$ (relative to $P \subset \hilbertG$) if there exist positive numbers  $l$ and $\delta $ such that
	\begin{equation}\label{eq:1.2}
	\dist({x^0},F(p)) \le l\| {p - {p^0}} \|\quad \forall p \in {V_\delta }({p^0})\quad     (\forall p \in {V_\delta }({p^0}) \cap P).
	\end{equation}
\end{definition}
Note that \eqref{eq:1.2} implies that  $F(p) \cap {V_\varepsilon }({x^0}) \ne \emptyset $ for any $\varepsilon  > 0$, if $\| {p - {p^0}} \| < \min \{ \delta ,\varepsilon /l\} $.

\begin{definition}   A set-valued mapping $F$ is Lipschitz-like (pseudo-Lipschitzian) (relative to $P \subset \hilbertG$) at  $({p^0},{x^0}) \in grF$ (where ${p^0} \in P$)  if there exist a number ${l_F} > 0$  and  neighbourhoods $V({p^0})$ and $V({x^0})$ such that   $$
	F({p^1}) \cap V({x^0}) \subset F({p^2}) + {l_F}\| {{p^2} - {p^1}} \|B
	$$                                                                                 
	for all  ${p^1},{p^2} \in V({p^0})$ (${p^1},{p^2} \in V({p^0}) \cap P$).
\end{definition}

Let $P \subset \hilbertG$, $X \subset \hilbertH$  and $I(p,x): = \{ i \in I\ | \ h_i(p,x) = 0\}$ be the set of indices of active inequality constraints at $(p,x)\in grF$ . Following \cite{MR2801389,MR2783218} %[10,11] 
we define the relaxed constant rank constraint qualification (RCRCQ) which generalizes the  constant rank constraint qualification   introduced by Janin \cite{MR751246}.

\begin{definition}\label{def:4} The set-valued mapping $F$ satisfies the Relaxed Constant Rank Constraint Qualification, or shortly, RCRCQ (relative to $P \times X$) at  $({p^0},{x^0}) \in grF$,  if for any index set $K \subset I(p{^0},{x^0})$ 
	$$
	\rank\{ {\nabla _x}{h_i}(p,x):{\rm{  }}i \in {I_0} \cup K\} =
	\rank\{ {\nabla _x}{h_i}(p^0,x^0):{\rm{  }}i \in {I_0} \cup K\}
	$$ 
	in a neighbourhood of $(p{^0},{x^0})$ (for $(p,x) \in P \times X$ from this neighbourhood). 
	
	The set $F(p^0)$ satisfies RCRCQ at  $x^0 \in F(p^0)$  if for any index set $K \subset I(p^0,x^0)$ 
	$$
	\rank\{ {\nabla _x}{h_i}(p^0,x):{\rm{  }}i \in {I_0} \cup K\} =
	\rank\{ {\nabla _x}{h_i}(p^0,x^0):{\rm{  }}i \in {I_0} \cup K\}
	$$ 
	for all $x$ in a neighbourhood of $x^0$. 
\end{definition}
%It follows immediately from Definition \ref{def:4}  that,
Clearly, if the set-valued mapping $F$ satisfies RCRCQ at  $({p^0},{x^0}) \in grF$, then it satisfies RCRCQ at all points $(p,x) \in gr\, F$ in some neighbourhood of $({p^0},{x^0})$. 

The following lemma proves the equality  $\Gamma (F(p),x) = T(F(p),x)$ under RCRCQ. In the finite dimensional case, where $\hilbertH=\mathbb{R}^s$ this fact has been proved in Theorem 1 of \cite{MR2801389}. In the infinite-dimensional case considered in the present paper this fact has been proved in Theorem 6.3 of \cite{on_tangent_cone_Bednarczuk}.

\begin{lemma}\label{lemma:1.1}(\cite{on_tangent_cone_Bednarczuk,MR2801389}).    Let the set-valued mapping $F$ satisfy RCRCQ (relative to $dom F \times \hilbertH$) at   $({p^0},{x^0}) \in grF$. Then there exist neighborhoods $V({p^0})$ and $V({x^0})$ such that $\Gamma (F(p),x) = T(F(p),x)$  for all $x \in F(p) \cap V({x^0})$ and all $p \in V({p^0}) \cap domF$. 
\end{lemma}
\begin{proof} As already noted, if $F$ satisfies RCRCQ at $(p^0,x^0)$, there are neighbourhoods $V({p^0})$ and $V({x^0})$ such that $F$ satisfies RCRCQ at any point $(p,x)$ where $p \in V({p^0}) \cap domF$, $x \in F(p) \cap V({x^0})$. Hence, the set $F(p)$ satisfies  RCRCQ at $x$ and by Theorem 6.3 of \cite{on_tangent_cone_Bednarczuk}, %[10, Theorem 1]
	$\Gamma (F(p),x) = T(F(p),x)$.
\end{proof}
Following \cite{fedorov1979numerical,MR1946832} %[4,8]
we define the  R-regularity of set-valued mappings.

\begin{definition}
	\label{definition_R_regularity}
	The set-valued mapping $F$ is R-regular at  $({p^0},{x^0}) \in grF$ (relative to $P \subset \hilbertG$)  if there exist a number $M > 0$ and neighbourhoods $V({p^0})$ and  $V({x^0})$ such that 
	\begin{align}
	\begin{aligned}\label{r-regularity}
	&\dist(x,F(p)) \le M\max \{ 0,{h_i}(p,x),\ i \in I, \ | {{h_i}(p,x)} |,\ i \in {I_0}\}\\
	&\text{for all}\ x \in V({x^0})\ \text{and}\  p \in V({p^0})\quad (p \in V({p^0}) \cap P).
	\end{aligned}
	\end{align}                                                                                  
\end{definition}

The concept of R-regularity appears in different works (see e.g. Theorem 2.84 and formula (2.164) of \cite{Bonnans_Shapiro}, formula (10) of \cite{aubin_criterion_for_metric_regularity}). When $h_i(x,p)=g_i(x)-p_i$, $g_i(p):\ \hilbertG \rightarrow \mathbb{R}$, $p_i\in \mathbb{R}$ $i\in I\cup I_0$ the R-regularity is equivalent to the metric regularity of $G(x,p):=[g_i(x)-p_i]_{i\in I\cup I_0}  -K$, where $K=\mathbb
{R}_{-}^{m}\times \{0\}^{n-m}$ (see formulas (2.143), (2.144) of \cite{Bonnans_Shapiro}). In the paper \cite{Kruger2015}, some variants of \eqref{definition_R_regularity} have been investigated (see e.g. formula (6) of \cite{Kruger2015}).

\section{Criterion of R-regularity}
Let  $v \in \hilbertH$ and  $\Pi_{F(p)}(v):=\{\tilde{v}\in F(p)\mid \|\tilde{v}-v\|=\dist(v, F(p)) \}$.  The set $\Pi _{F(p)}(v)$  is the  solution set to the problem 
\begin{equation}\label{problem:P1}\begin{array}{ll}
\text{minimize}& {f_v}(p,x) = \| {x - v} \|\\
\text{subject to} & x \in F(p).
\end{array}
\end{equation}
The problem \eqref{problem:P1} can be  equivalently reformulated as
\begin{equation}\label{problem:P2}\begin{array}{ll}
\text{minimize}& \frac{1}{2}\| {x - v} \|^2\\
\text{subject to} & x \in F(p).
\end{array}
\end{equation}
Lagrange multiplier sets for problem \eqref{problem:P1} are defined as follows
\begin{align*}
&\Lambda_v(p,x):=\{ \lambda\in \mathbb{R}^n \mid \frac{x-v}{\|x-v\|}+\sum_{i=1}^{n} \lambda_i \nabla_x h_i(p,x)=0,\ \lambda_i\geq 0,   \lambda_i h_i(p,u)=0\ \text{for}\ i\in I\}, \\
&\Lambda_v^M(p,x):=\{\lambda \in \Lambda_v(p,x) \mid \sum_{i=1}^{n} |\lambda_i|\leq M  \}.
\end{align*}
\begin{lemma}\label{lemma:multipliers} (Proposition 7.1 of \cite{on_tangent_cone_Bednarczuk}, Theorem 1 of \cite{MR2801389}). Suppose that $x^0\in \hilbertH$ is a solution to \eqref{problem:P1} for $p=p^0\in \hilbertG$ and $v\in \hilbertH$. Assume the set-valued mapping $F$ satisfies RCRCQ (relative to $dom F \times \hilbertH$) at   $({p^0},{x^0}) \in grF$. Then $\Lambda_v(p^0,x^0)\neq \emptyset$. 
\end{lemma}
The following theorem generalizes   Theorem 2 \cite{MR2783218} and Theorem 4.1 \cite{Guo2014} to parametric systems defined by the set-valued mapping $F$.  
\begin{theorem}\label{theorem:2.1} Let $({p^0},{x^0}) \in grF$ and  the set-valued mapping $F$ be l.s.c. at  $({p^0},{x^0})$ relative to $domF$.  The following assertions are equivalent:
	\begin{enumerate}
		\item[(a)] the set-valued mapping $F$ is R-regular at   $({p^0},{x^0})$ relative to $domF$;
		\item[(b)] there exists a number $M > 0$ such that  for any sequences ${p^k} \to {p^0},{p^k} \in domF,$   ${v^k} \to {x^0},{v^k} \notin F({p^k})$,  the inequality $\Lambda _{v_k}^M({p^k},{x^k}) \ne \emptyset $ holds for all 
		${x^k} = x({p^k},{v^k}) \in \Pi_{F({p^k})}({v^k})$ and $k$ sufficiently large.
	\end{enumerate}
	
\end{theorem}

\begin{proof}If  ${x^0} \in {\mathop{\rm int}} F({p^0})$,  the theorem is obviously valid.   Assume that ${x^0} \in bdF({p^0})$.
	\begin{enumerate}
		\item[1)] The implication $(a) \Rightarrow (b)$ follows from a slight modification in the first part of the proof of Theorem 2 \cite{MR2783218}. 
		\item[2)]  $(b) \Rightarrow (a)$. On the contrary, suppose that $F$ is not R-regular relative  $domF$ at $({p^0},{x^0})$. Then there exist sequences ${p^k} \to {p^0},{p^k} \in domF,$ and  ${v^k} \to {x^0},{v^k} \notin F({p^k})$, such that for all $k = 1,2,\dots$
		\begin{equation}\label{eq:2.1}
		\dist({v^k},F({p^k})) > k\max \{ 0,h_i({p^k},{v^k}),\ i \in I,\ | {{h_i}({p^k},{v^k})} |,\ i \in {I_0}\} .     
		\end{equation}
		Take any ${x^k} \in \Pi_{F({p^k})}({v^k})$. Due to (b) there exists a vector ${\lambda ^k}$ such that  $\sum\limits_{i = 1}^n {\left| {\lambda _i^k} \right|}  \le M$ and
		\begin{equation}\label{eq:2.2}
		\begin{array}{l}
		\frac{x^k-v^k}{\|x^k-v^k\|}+\sum_{i=1}^{n} \lambda_i^k \nabla_xh_{i}(p^k,x^k)=0,\ \lambda_i^k \nabla_xh_{i}(p^k,x^k)=0,\\ \lambda_i^k\geq 0\ i \in I(p^k,x^k),\ \lambda_i^k=0\ i\in I\setminus I(p^k,x^k).
		\end{array}
		\end{equation}
		It follows from the lower semicontinuity of $F$ at $(p^0,x^0)$ that there exists a sequence ${q^k} \in F({p^k})$ such that ${q^k} \to {x^0}$. Then $\| {{v^k} - {x^k}} \| \le \| {{v^k} - {q^k}} \|$ and, therefore, ${x^k} \to {x^0}$. 
		
		In virtue of the boundedness of the sequence ${\lambda ^k}$ and of the condition $\lambda _i^k{h_i}({p^k},{x^k}) = 0$, for $i \in I$, by \eqref{eq:2.2}, for all sufficiently large  $k$ we obtain
		\begin{align*}
		&\|x^k-v^k\|=\langle \sum_{i=1}^{n} \lambda_i^k \nabla_x h_i(p^k,x^k),v^k-x^k\rangle\\
		& \leq \sum_{i=1}^n \lambda_i^k (h_i(p^k,v^k)-h_i(p^k,x^k)+o(\|v^k-x^k\|))=\\
		& = \sum\limits_{i = 1}^n {\lambda _i^k{h_i}({p^k},{v^k}) + \sum\limits_{i = 1}^n {\lambda _i^k} o(\| {{v^k} - {x^k}} \|)) \le } \sum\limits_{i = 1}^n {\lambda _i^k{h_i}({p^k},{v^k}) + } \frac{1}{2}\| {{v^k} - {x^k}} \|.
		\end{align*}
		
		The latter inequality implies	 
		$$\dist({v^k},F({p^k})) = \| {{x^k} - {v^k}} \| \le 2M\max \{ 0,{h_i}{({p^k},{v^k})},\ i \in I,\ | {{h_i}({p^k},{v^k})} |,\ i \in {I_0}\}, $$
		which contradicts  \eqref{eq:2.1}.    Thus $(b) \Rightarrow (a)$.    
	\end{enumerate}
\end{proof}
\begin{remark} Theorem \ref{theorem:2.1} is a generalization of Theorem 2 \cite{MR2783218}. Unlike Theorem 2 \cite{MR2783218} it does not require  Lipschitz continuity of gradients ${\nabla _x}{h_i}(p,x)$ for $i = 1,\dots,n$. Theorem \ref{theorem:2.1} considers  also more general notion of R-regularity relative $domF$.
\end{remark}
\begin{remark} As follows from the proof of Theorem 2 \cite{MR2783218} the implication $(a) \Rightarrow (b)$ holds without the assumptions of lower semicontinuity of $F$ at $({p^0},{x^0})$.
\end{remark}
\begin{remark}
	Let us note that Lemma 3 of \cite{MR2783218} is a consequence of Theorem \ref{theorem:2.1} and Theorem \eqref{theorem:3.1} below, which says that RCRCQ for $F$ at $(p^0,x^0)$ and lower Lipschitz continuiuty of $F$ at $(p^0,x^0)$ implies $R$-regularity of $F$ at $(p^0,x^0)$. 
	
	In \cite{2018arXiv181105166B} we discussed in details the proof of Lemma 3 of \cite{MR2783218} for some special functions $h_i$, $i=1,\dots,n$. The proof of Theorem \ref{theorem:2.1} together with Theorem \ref{theorem:3.1} fill some gaps in the proof of Lemma 3 of \cite{MR2783218}.
\end{remark}

The example below shows that the assertion of Theorem \ref{theorem:2.1} may not hold   if $F$ is not lsc at a point $({p^0},{x^0})$.

\begin{example} Let $F(p) = \{ x \in {\mathbb{R}^2}\mid  {\| {{x_1}} \| \le {{1,}}\ \| {{x_2}} \| \le {{1,}}\ {x_2} - p{x_1} + | p | + 1 \le 0\} }$, $p \in \mathbb{R}$.  
	Then $F(p) = \{ (1, - 1)\} $ for all $p > 0$,  $F(p) = \{ ( - 1, - 1)\} $ for $p < 0$ and $F(p) = \{ x \in \mathbb{R}^2 \mid  { - 1 \le {x_1}}  \le {1,\ }{x_2} =  - 1\} $ for $p = 0$.   Consider the point $({p^0},{x^0})$, where ${p^0} = 0$,  ${x^0} = (0, - 1)$. 
	
	Evidently, $F$ is not lsc at  $({p^0},{x^0})$.  Let us take $p = \varepsilon  > 0$, $v = (\varepsilon , - 1)$, where $\varepsilon  \to 0$. 
	It is easy to see that for given $v$ and $p$ the R-regularity condition does not hold if $\varepsilon $ is sufficiently small. 
\end{example}

The following technical observation will be used in the sequel.

\begin{proposition}\label{propostion:I_1-linearly_independent}
	Let $p^0\in \setD$. Assume that RCRCQ holds for the set-valued mapping $F$  given by \eqref{eq:1.1} at $(p^0,x^0)\in \gph F$ and $F(p)\neq \emptyset$ for $p\in V_0(p^0)$.  Then there exist neighbourhoods $V(p^0)$, $V(x^0)$ and an index set $I_0^\prime\subset I_0$, $|I_0^\prime|=\rank \{ \nabla_x h_i(p_0,x_0)\ i\in I_0 \}$ such that for all $(p,x)\in V(p^0)\times V(x^0)$ vectors
	$\nabla_x h_i(p,x)$, $i\in I_0^\prime$ are linearly independent. 
\end{proposition}
\begin{proof}
	The assertion is valid if $\nabla_x h_i(p^0,x^0)$, $i\in I_0$ are linearly independent.
	Suppose that $\nabla_x h_i(p_0,x_0)$, $i\in I_0$ are linearly dependent. By RCRCQ there exist neighbourhoods $V_0(p^0)$, $V_0(x^0)$ such that 
	\begin{equation*}
	\rank \{ \nabla_x h_i (p,x),\ i\in I_0  \}=\rank \{ \nabla_x h_i (p^0,x^0),\ i\in I_0   \}.
	\end{equation*}
	Let $\rank \{ \nabla_x h_i (p^0,x^0)  ,\ i\in I_0 \}=k$. Then there exists indices $i_1,\dots,i_k\subset I_0$, $i_j\neq i_k$ for $j\neq k$ such that $\nabla_x h_{i_1} (p^0,x^0),\dots,\nabla_x h_{i_k} (p^0,x^0)$ are linearly independent.
	Denote $I_0^\prime =\{i_1,\dots,i_k\}$. Then, by the continuity of gradients of $h_i$, $i=1,\dots,n$ with respect to variable $x$, $\nabla_x h_i(\cdot,\cdot)$, $i\in I_0^\prime$ are linearly independent in some neighbourhood of $(p^0,x^0)$.
	%		 By continuity of $\nabla_x h (\cdot,\cdot)$ there exist  neighbourhoods $U_1(p^0)$, $V_1(x^0)$ such that for any $(p,x)\in U_1(p^0),V_1(x^0)$, vectors $\nabla_x h_{i_1} (p,x),\dots,\nabla_x h_{i_k} (p,x)$ are linearly independent. Let us note that by RCRCQ and the continuity of $\nabla h_i(\cdot,\cdot)$, $i\in I_1$, for all $p\in U_0(p^0)\cap U_1(p^0)$, $x\in V_0(x^0)\cap V_1(x^0)$
	%		\begin{align*}
	%		\rank \{ \nabla_x h_i (p,x),\ i\in I_0  \}=	\rank \{ \nabla_x h_{i_1} (p,x),\dots,\nabla_x h_{i_k} (p,x)   \}=k.
	%		\end{align*}
	%		By Theorem (...) we have that there exist smooth functions $g_l:\ \mathbb{R}^k\rightarrow \mathbb{R}$, $l\in I_0\setminus I_0^\prime$ such that for any $(\hat{p},\hat{x})\in  (U_0(p^0)\cap U_1(p^0))\times (V_0(x^0)\cap V_1(x^0))$ 
	%		\begin{align*}
	%		&	\forall (p,x)\in (U_0(p^0)\cap U_1(p^0))\times (V_0(x^0)\cap V_1(x^0))\\
	%		&	h_l(p,x)=g_l(h_{i_1}(p,x),\dots,h_{i_k}(p,x))\quad l\in I_0\setminus I_0^\prime.
	%		\end{align*}
	%		Let $(p,x)\in (U_0(p^0)\cap U_1(p^0))\times (V_0(x^0)\cap V_1(x^0))$. Then
	%		\begin{align*}
	%		&h_i(p,x)=0,\ i\in I_0^\prime\\
	%		&\implies h_i(p,x)=0,\ i\in I_0^\prime \wedge h_l(p,x)=g_l(h_{i_1}(p,x),\dots,h_{i_k}(p,x))\ l\in I_0\setminus I_0^\prime\\
	%		&\implies h_i(p,x)=0,\ i\in I_0^\prime \wedge h_l(p,x)=g_l(0,\dots,0)\ l\in I_0\setminus I_0^\prime\\
	%		&\implies h_i(p,x)=0,\ i\in I_0^\prime \wedge h_l(p,x)=0\ l\in I_0\setminus I_0^\prime
	%		\end{align*}
	%		and $h_i(p,x)$, $i\in I_0^\prime$ are functionally independent. 
\end{proof}

In view of Proposition \ref{propostion:I_1-linearly_independent}, RCRCQ implies that there exists a subset  $I_0^\prime\subset I_0$ of indices of parametric system defined by the set-valued mapping $F$ such that\begin{equation*}
\rank \{ \nabla_x h_i(p,x), i\in I_0^\prime  \}= \rank \{ \nabla_x   h_i(p^0,x^0), i\in I_0^\prime \}=|I_0^\prime|,
\end{equation*}
for $(p,x)$ in some neighbourhood of $(p^0,x^0)$.

\section{Relaxed constant rank condition and R-regularity}

It is known \cite{MR825383,MR1946832} %[2,8] 
that the Mangasarian-Fromovitz constraint qualification (MFCQ) \cite{MR0207448} for the  set $F({p^0})$ at a point ${x^0} \in F({p^0})$ implies   R-regularity of the set-valued mapping $F$ at $({p^0},{x^0}) \in grF$. 

We show that RCRCQ implies R-regularity of the set-valued mapping $F$.
\begin{theorem}\label{theorem:3.1} Assume that 1) $F$ is lsc at $({p^0},{x^0}) \in grF$ relative to $domF$;  \\
	2) $F$satisfies RCRCQ at $({p^0},{x^0}) \in grF$ relative to $domF \times \hilbertH$. 
	
	Then $F$ is R-regular at $(p{,^0}{x^0})$ relative to $domF$.
\end{theorem}
\begin{proof}By Theorem \ref{theorem:2.1}, R-regularity  of the mapping $F$ at $({p^0},{x^0}) \in grF$ is equivalent to the fact that there exists a number $M > 0$ such that  for any sequences ${p_k} \to {p^0},{p_k} \in domF,$   ${v_k} \to {x^0},{v_k} \notin F({p_k})$,  the inequality $\Lambda _v^M({p_k},{x_k}) \ne \emptyset $ holds for all 
	${x_k} = x({p_k},{v_k}) \in \Pi _{F({p_k})}({v_k})$ and $k$ sufficiently large.

	On the  contrary, suppose that there exist sequences $p_k\rightarrow p^0$, $v_k\rightarrow x^0$, $x_k\in F({p_k})$ such that $v_k\notin F(p_k)$, $x_k\in \Pi _{F({p_k})}({v_k})$ and
	\begin{equation}\label{assumption:unbounded}
	\dist (0,\Lambda_{v_k}(p_k,x_k))\rightarrow+\infty.
	%			\dist({v_k},F({p_k})) > k\max \{ 0,{h_i}({p_k},{v_k}),\ i \in I,| {{h_i}({p_k},{v_k})} |,\ i \in {I_0}\} ,\quad  k = 1,2,\dots .
	\end{equation} 
	
	Due to the fact that $x^0\in \liminf\limits_{p\rightarrow p^0} F(p),$ without loss of generality, we can assume that $F(p_k)\neq \emptyset$ for each $p_k$, $k=1,\dots$, and for any $\hat{x}_k\in \Pi_{F(p_k)}(v_k)$ we have $\hat{x}_k\rightarrow x^0$. In consequence, $x_k\rightarrow x_0$.
	%			By \eqref{assumption:unbounded} we have
	%			\begin{equation*}
	%				\dist(v_k,x_k) > k\max \{ 0,{h_i}({p_k},{v_k}),\ i \in I,| {{h_i}({p_k},{v_k})} |,\ i \in {I_0}\} ,\quad  k = 1,2,\dots .
	%			\end{equation*}
	%	 Then $\|x_k-w_k\|\leq \|w_k-\hat{x}_k\|\leq \|w_k-x_0\|+\|\hat{x}_k-\bar{x}\|$ and therefore, $x_k\rightarrow\bar{x}$.
	
	As already noted, if RCRCQ  holds at $(p^0,x^0)$, then  RCRCQ holds also at all  points close to $(p^0,x^0)$.  Without loss of generality, we can assume that RCRCQ holds at all $(p_k,x_k)$, $k=1,2,\dots$ . Consequently, by Lemma \ref{lemma:multipliers}, $\Lambda_{v_k}(p_k,x_k)\neq \emptyset$ for all $k=1,2,\dots$. 
	
	Without loss of generality, we can assume that $(p_k,v_k)\in V(p^0,x^0)$, where by RCRCQ, $V(p^0,x^0)$ is such that for any $J$, $I_0\subset J\subset I_0\cup I(p_0,x_0)$
	\begin{equation}\label{conditon:RCRCQ}
	\rank\{ \nabla _x h_i(p,x),\ i\in J\} =\rank\{\nabla_x h_i(p^0,x^0),\ i\in J\}\quad \forall (p,x)\in V(p^0,x^0). 
	\end{equation}
	By Lemma \ref{lemma:1.1}, $\Gamma(F(p^k),x^k)=T(F(p^k),x^k) $ and by the necessary optimality conditions for  problem \eqref{problem:P2}\footnote{In the literature it is often assumed that Robinson constraint qualification holds (see for example \cite{Bonnans_Shapiro}). However, it is enough to assume that the $T(F(p^k),x^k)$ coincides with feasible set to the linearized problem to \eqref{problem:P2} (see discussion after Lemma 3.7 of \cite{Bonnans_Shapiro}).} we have
	%Theorem \ref{theorem:representation} from the Appendix,
	\begin{equation}\label{eq:representation_lambda}
	v_k-x_k=\sum_{i\in I_0\cup I(p_k,x_k)} \hat{\lambda}_i^k \nabla_x h_i(p_k,x_k),\quad k=1,\dots
	\end{equation}
	where $\hat{\lambda}_i^k\in \mathbb{R}$, $i\in I_0$ and $\hat{\lambda}_i^k\geq 0$, $i\in I(p_k,x_k)$, $k=1,\dots$. 
	Recall that $I(p,x):=\{ i\in I \mid  h_i(p,x)=0  \}$ and $\hat{\lambda}_i^k$, $i\in I\cup I(p_k,x_k)$, are related to the set $\Lambda_{v_k}(p_k,x_k)$ via the
	relationship 
	\begin{align*}%\label{equivalence:relations_lambda}
	&\frac{v_k-x_k}{\|v_k-x_k\|}=\sum_{i\in I_0\cup I(p_k,x_k)}\lambda_i^k \nabla_x h_i(p_k,x_k)\iff v_k-x_k=\sum_{i\in I_0\cup I(p_k,x_k)}\hat{\lambda}_i^k  \nabla_x h_i(p_k,x_k),
	\end{align*}
	for some $\lambda^k\in \Lambda_{v_k}(p_k,x_k)$.
	%				 \eqref{equivalence:relations_lambda}. 
	%				Then \eqref{eq:representation_lambda} takes the form
	%			\begin{align}
	%			\begin{aligned}\label{eq:representation_lambda1}
	%			& v_k-x_k=\sum_{i\in I_1} \hat{\lambda}_i^k g_i(p_k)+\sum_{i\in I_{p_k}(P_{\multifC(p_k)}(w_k)\setminus I_1} \hat{\lambda}_i^k g_i(p_k),\\
	%			& \hat{\lambda}_i^k\geq 0, i\in I_{p_k}(P_{\multifC(p_k)}(w_k)\setminus I_1\quad k=1,\dots
	%			\end{aligned}
	%			\end{align}
	
	By Proposition \ref{propostion:I_1-linearly_independent}, there exist neighbourhoods $V_1(p^0)$, $V_1(x^0)$ and indices $I_0^\prime\subset I_0$, 
	$|I_0^\prime|=\rank \{ \nabla_x h_i(p_0,x_0)\ i\in I_0^\prime \}$ such that for all $(p,x)\in V_1(p^0)\times V_1(x^0)$ vectors
	$\nabla_x h_i(p,x)$, $i\in I_0^\prime$ are linearly independent. Hence, without loss of generality we can assume that $(p_k,x_k)\in V_1(p^0)\times V_1(x^0)$, and for every $k=1,\dots$, we have
	\begin{equation*}
	\sum_{i\in I_0} \hat{\lambda}_i^k \nabla_x h_i(p_k,x_k)=\sum_{i\in I_0^\prime} \check{\lambda}_i^k \nabla_x h_i(p_k,x_k),
	\end{equation*}
	for some $\check{\lambda}_i^k\in \mathbb{R}$, $i\in I_0^\prime$. 
	
	Hence we can rewrite \eqref{eq:representation_lambda} as follows
	%			
	%			there exists $I_0^\prime\subset I_0$, such that 
	%			
	%			 $I_1^0(p_k,x_k)\subset I_1(p_k,x_k)$,  and $\tilde{\lambda}_i(v_k,p_k)\in \mathbb{R}$, $i\in I^0$, $\tilde{\lambda}_i(v_k,p_k)>0$, $i\in I_1^0(p_k,x_k)$, such that	
	\begin{equation}\label{eq:representation_sequence}
	v_k-x_k=\sum_{i \in I_0^\prime} \check{\lambda}_i^k \nabla_x h_i(p_k,x_k)+\sum_{i\in I(p_k,x_k)} \hat{\lambda}_i^k \nabla_x h_i(p_k,x_k),
	\end{equation}
	where $\nabla_x h_i(p_k,x_k)$, $i\in I_0^\prime$, are linearly independent.
	
	Passing to a subsequence, if necessary, we may assume that for all $k\in \mathbb{N}$, $ I(p_k,x_k)$ is a fixed set, i.e., $ I(p_k,x_k)=I^0$.
	
	By \cite[Lemma 2]{2018arXiv181105166B} we have that for any $k=1,\dots$ there exists $I^k\subset I^0$ such that 
	\begin{equation*}
	\sum_{i \in I_0^\prime} \check{\lambda}_i^k \nabla_x h_i(p_k,x_k)+\sum_{i\in I^0} \hat{\lambda}_i^k \nabla_x h_i(p_k,x_k)=\sum_{i \in I_0^\prime} \tilde{\lambda}_i^k \nabla_x h_i(p_k,x_k)+\sum_{i\in I^k} \tilde{\lambda}_i^k \nabla_x h_i(p_k,x_k),
	\end{equation*}
	where $\tilde{\lambda}_i^k\geq 0$, $i\in I^k$ and $\nabla_x h_i(p_k,x_k)$, $i\in I_0^\prime \cup I^k$ are linearly independent. Hence we can rewrite \eqref{eq:representation_sequence} as 
	\begin{equation}\label{eq:representation_sequence2}
	v_k-x_k=\sum_{i \in I_0^\prime} \tilde{\lambda}_i^k \nabla_x h_i(p_k,x_k)+\sum_{i\in I^k} \tilde{\lambda}_i^k \nabla_x h_i(p_k,x_k).
	\end{equation}
	Again, passing to a subsequence, if necessary, we may assume that $I^k=I^\prime$ is a fixed set.
	
	%			By RCRCQ, there exists $k_0$ such that for all $k\geq k_0$
	%			\begin{equation*}
	%			\rank\{ \nabla_x h_i(p_k,x_k), \ i\in I^0\cup I_1^0 \}=\rank\{\nabla_x h_i(p^0,x^0), \ i\in I^0\cup I_1^0 \}.
	%			\end{equation*}
	Put $\lambda_i^k=\tilde{\lambda}_i^k \|v_k-p_k\|^{-1}$, $i\in I_0^\prime\cup I^\prime$, $k=1,\dots$ and $\lambda_i^k=0$, $i\in I_0\cup I \setminus (I_0^\prime \cup I^\prime)$. %, where $\tilde{\lambda}^k=[\tilde{\lambda}_i^k]_{i\in I^0\cup I_1^0}$. 	
	Let us denote $\lambda^k=[\lambda_i^k]_{i\in I_0^\prime\cup I^\prime}$. We have  $\lambda^k\in \Lambda_{v_k}(p_k,x_k)$ 		
	and, by  \eqref{assumption:unbounded}, $\|\lambda^k\|\rightarrow+\infty $. Without loss of generality we may assume that $\lambda^k
	\|\lambda^k\|^{-1}\rightarrow\bar{\lambda}$.  From \eqref{eq:representation_sequence2} we have
	\begin{equation}\label{eq:representation_sequence3}
	\frac{v_k-x_k}{\|v_k-x_k\|\|\lambda^k\|}=\sum_{i \in I_0^\prime} \frac{\lambda_i^k}{\|\lambda_k\|} \nabla_x h_i(p_k,x_k)+\sum_{i\in I^k} \frac{\lambda_i^k}{\|\lambda_k\|} \nabla_x h_i(p_k,x_k).
	\end{equation}
	By passing to the limit in \eqref{eq:representation_sequence3} we obtain
	\begin{equation*}
	0=\sum_{i\in I_0^\prime\cup I^\prime} \bar{\lambda}_i \nabla_x h_i(p^0,x^0),\ \bar{\lambda}_i\geq 0,\ i\in I^\prime, 
	\end{equation*}
	where $\|\bar{\lambda}\|=1$. This contradicts the fact that for $k=1,\dots$
	\begin{align*}
	&\rank \{\nabla_x h_i(p_0,x_0),\ i\in I_0^\prime\cup I^\prime  \}\\
	&=\rank \{\nabla_x h_i(p_0,x_0),\ i\in I_0\cup I^\prime  \}\\
	&=\rank \{\nabla_x h_i(p_k,x_k),\ i\in I_0\cup I^\prime  \}\\
	&=\rank \{\nabla_x h_i(p_k,x_k),\ i\in I_0^\prime\cup I^\prime  \}=|I_0^\prime\cup I^\prime|,
	\end{align*}
	i.e. vectors $\nabla_x h_i(p_0,x_0),\ i\in I_0^\prime\cup I^\prime$ are linearly independent.
	%			$\nabla_x h_i(p^0,x^0)$, $i\in I_1^0\cup I_2^0$, are linearly independent.
\end{proof}

\section{Lipschitz-likeness of $F$}

The following theorem provides the relationships between R-regularity and pseudo-Lipschitzness of the set-valued mapping $F$ defined on a normed space $\hilbertG$.
\begin{theorem}\label{theorem:3.3}
	Assume that $F$ is R-regular at a point $({p^0},{x^0}) \in grF$ relative to $domF$. Then $F$ is pseudo-Lipschitzian at this point relative to $domF$.\end{theorem}
\begin{proof} If  $F$ is R-regular at a point $({p^0},{x^0}) \in grF$ relative to $domF$,  this means that there are  numbers $M > 0$,$\delta  > 0$, $\varepsilon  > 0$ such that  
	\begin{equation}\label{eq:3.7}
	d(x,F(p)) \le M\max \{ 0,{h_i}(p,x):i \in I,\left| {{h_i}(p,x)} \right|:i \in {I_0}\}     
	\end{equation}                                                                 
	for all $x \in {V_\varepsilon }({x^0})$ and all $p \in {V_\delta }({p^0}) \cap domF$.
	
	Denote    $l = \max \{ {l_i}\left| {i = 1,\dots,p\} } \right.$ where ${l_i}$ are Lipschitz constants for functions ${h_i}(p,x)$ on some set $V({p^0}) \times V({x^0})$.  Choose numbers $\delta $ and $\varepsilon $ such that  ${V_\delta }({p^0}) \subset V({p^0}),{V_\varepsilon }({x^0}) \subset V({x^0})$.  Then
	\begin{align*}
	&d({x^0},F(p)) \le M\max \{ 0,{h_i}(p,{x^0}):i \in I,\left| {{h_i}(p,{x^0})} \right|:i \in {I_0}\}  \le \\
	& \le M\max \{ 0,{h_i}(p,{x^0}) - {h_i}({p^0},{x^0}):i \in I,\left| {{h_i}(p,{x^0}) - {h_i}({p^0},{x^0})} \right|:i \in {I_0}\}  \le l\| {p - {p^0}} \|       
	\end{align*}                                                               
	for all $p \in {V_\delta }({p^0}) \cap domF$.
	
	This means that $F$ is Lipschitz lower continuous at $({p^0},{x^0})$ relative to $domF$ and, consequently (see \eqref{eq:1.2}),  $F(p) \cap {V_\varepsilon }({x^0}) \ne \emptyset $ for all $p \in {V_\eta }({p^0}) \cap domF$, where $\eta  < \min \{ \delta ,\varepsilon /l\} $.  Let $p,\tilde p \in {V_\eta }({p^0}) \cap domF$ and let $\tilde x \in F(\tilde p) \cap {V_\varepsilon }({x^0})$.  Then, from \eqref{eq:3.7} follows  
	\begin{align*}
	&d(\tilde x,F(p)) \le M\max \{ 0,{h_i}(p,\tilde x):i \in I,\left| {{h_i}(p,\tilde x)} \right|:i \in {I_0}\}  \le \\                                                                                     
	& \le M\max \{ 0,{h_i}(p,\tilde x) - {h_i}(\tilde p,\tilde x):i \in I,\left| {{h_i}(p,\tilde x) - {h_i}(\tilde p,\tilde x)} \right|:i \in {I_0}\}  \le l\| {p - \tilde p} \|.
	\end{align*}
	The last inequality is equivalent to 
	$$F(\tilde p) \cap {V_\varepsilon }({x^0}) \subset F(p) + l\| {p - \tilde p} \|B\ \text{for all}\  p,\tilde p \in {V_{{\delta _0}}}({p^0}) \cap domF.$$
\end{proof}

\begin{remark} Let us note that in Theorem \ref{theorem:3.3}  we only need to assume that all the functions $f$ and $h_i$, $i=1,\dots,n$  are locally Lipschitz continuous (may not be differentiable) near $({p^0},{x^0}) \in grF$. 
\end{remark}

By Theorem \ref{theorem:3.1} we obtain the following result

\begin{theorem}\label{theorem:pseudo_Lipschitzian}
	Assume that 1) $F$ is lsc at $({p^0},{x^0}) \in grF$ relative to $domF$;  \\
	2) $F$satisfies RCRCQ at $({p^0},{x^0}) \in grF$ relative to $domF \times \hilbertH$.
	
	Then $F$ is pseudo-Lipschitzian at this point relative to $domF$.	
\end{theorem}
Let us note that for some particular functions $h_i(\cdot,\cdot)$, $i\in I\cup I_0$ Theorem \ref{theorem:pseudo_Lipschitzian} has been already proved in \cite{2018arXiv181105166B}. 

In the finite-dimensional setting, when both $\hilbertH$ and $\hilbertG$ are finite-dimensional spaces the results analogous to Theorem \ref{theorem:pseudo_Lipschitzian} can be obtained via properties of the optimal value function defined as
$\varphi (p) := \inf \{ f(p,x)\mid x \in F(p)\} $ and the solution 
set $S(p): = \{ x \in F(p)\mid {f(p,x) \le \varphi (p)} \} $.
\begin{definition} Let ${p^0} \in domF$.  A mapping $F$ is locally bounded at ${p^0}$ if there exist a neighborhood $V({p^0})$ and a bounded set ${Y_0}$ such that $F(p) \subset {Y_0}$ for all $p \in V({p^0})$.
\end{definition}
\begin{theorem}\label{theorem:3.2} Assume that 1) F is locally bounded at ${p^0}$and  functions $f,{h_i}$ are Lipschitz continuous on a set $V({p^0}) \times ({Y_0} + \varepsilon B)$, where $\varepsilon  > 2diam{Y_0}$; \\
	2) $F$is R-regular at $({p^0},{x^0}) \in grF$ relative to $dom\, F$.\\
	Then $\varphi $ is Lipschitz continuous on some set $V({p^0}) \cap domF$.
\end{theorem}
\begin{proof}  Let ${l_0} > 0$  and ${l_i} > 0$ be  Lipschitz  constants  for  $f$ and ${h_i}$ on a set $V({p^0}) \times ({Y_0} + \varepsilon B)$. Then due to Lemma 3.68 \cite{MR1946832}
	$$\varphi (p) = \inf \{ f(p,x) + \alpha d(x,F(p))\mid {x \in {Y_0}}  + {2^{ - 1}}\varepsilon B\}\ \text{for all}\ p \in V({p^0})\ \text{and any}\ \alpha  > {l_0}.$$
	
	Let us take any $p,\tilde p \in V({p^0})$ and $\tilde x \in S(\tilde p)$. Then without loss of generality
	\begin{align*}
	&\varphi (p) - \varphi (\tilde p) \le f(p,\tilde x) + \alpha d(\tilde x,F(p)) - f(\tilde p,\tilde x) \le \\
	& \le f(p,\tilde x) - f(\tilde p,\tilde x) + \alpha M\max \{ 0,{h_i}(p,\tilde x):i \in I,\left| {{h_i}(p,\tilde x)} \right|:i \in {I_0}\}  \le \\
	& \le {l_0}\left| {p - \tilde p} \right| + \alpha M\max \{ 0,{h_i}(p,\tilde x) - {h_i}(\tilde p,\tilde x):i \in I,\left| {{h_i}(p,\tilde x) - {h_i}(\tilde p,\tilde x)} \right|:i \in {I_0}\}  \le \\
	& \le ({l_0} + \alpha M\max \{ {l_i}:i \in I \cup {I_0}\} )\| {p - \tilde p} \|.
	\end{align*}
	
	Similarly, one can obtain $\varphi (\tilde p) - \varphi (p) \le ({l_0} + \alpha M\max \{ {l_i}:i \in I \cup {I_0}\} )\| {p - \tilde p} \|.$       
\end{proof}

\section{Application to bilevel programming}
Consider a bilevel programming problem (BLPP):
\begin{align}\tag{BLPP}\label{BLPP}
\begin{array}{ll}
\text{minimize} &G(p,x), \\
\text{s.t.}&  x\in S(p):=\arg\min\{ f(p,x) \mid x\in F(p) \},\\
&p\in P := \{ p \in {\mathbb{R}^n}\mid  {{g_j}(p) \le {0}\ j \in J} \} ,
\end{array}
\end{align}
where $F(p)$ is defined in \eqref{eq:1.1}, $J = \{ 1,\dots,l\} $, functions $G(p,x)$, ${g_j}(p)$, $f(p,x)$ and ${h_i}(p,x)$ are continuously differentiable  (see e.g. monograph \cite{MR1921449}). 

The point $(p,x)$ is said to be a feasible point to the problem \eqref{BLPP} if $p \in P$, $x \in S(p)$. A feasible $({p^0},{x^0})$ is called a solution (local solution) of the problem \eqref{BLPP} if $G({p^0},{x^0}) \le G(p,x)$ for all feasible points $(p,x)$ (for all feasible points from some neighborhood of $({p^0},{x^0})$).

The problem \eqref{BLPP} can be equivalently reformulated as the following one-level problem
\begin{equation}\label{eq:4.1}
\begin{array}{ll}
\text{minimize} & G(p,x),\\
\text{s.t.} & p \in P,\ x \in S(p) = \{ x \in F(p)\mid f(p,x) \le \varphi (p)\},
\end{array}
\end{equation}
where $\varphi (p) = \inf \{ f(p,x)|x \in F(p)\} $ is the optimal value function of the lower-level problem. 
Main difficulty in solving  problem \eqref{eq:4.1} comes from the  nonsmoothness of the value function $\varphi (p)$.  Ye and Zhu \cite{doi:10.1080/02331939508844060} introduced the concept of partial calmness which allowed to move the nonsmooth constraint from the feasible set to the objective function.

Let $({p^0},{x^0})$ be a feasible point of the problem \eqref{BLPP}. The problem \eqref{BLPP} in the form \eqref{eq:4.1} is called partially calm at $({p^0},{x^0})$, if there exist a number $\mu  > 0$ and a neighborhood $V$ of the point $({p^0},{x^0},0)$ in ${\mathbb{R}^n} \times {\mathbb{R}^m} \times \mathbb{R}$ such that $G(p,x) - G({p^0},{x^0}) + \mu \left| u \right| \ge 0$ for all $(p,x,u) \in V$ such that $p \in \hilbertG$, $x \in S(p)$, $f(p,x) - \varphi (p) + u = 0$.

In \cite{doi:10.1080/02331939508844060} it was proved that the problem \eqref{BLPP} in the form \eqref{eq:4.1} is partial calm at its local solution $({p^0},{x^0})$if and only if there exists a number $\mu  > 0$ such that $({p^0},{x^0})$ is a local solution of the partially penalized problem
\begin{equation}\label{eq:4.2}
\begin{array}{ll}
\text{mimimize} & G(p,x) + \mu (f(p,x) - \varphi (p)), \\
\text{s.t.} & p \in P,\ x \in F(p).   \end{array}
\end{equation}
Ye and Zhu \cite{doi:10.1080/02331939508844060} showed that the problem \eqref{BLPP} with a linear in $(p,x)$ lower-level problem is partially calm. 

Let us define the set-valued mapping $S:p \rightrightarrows S(p)$ and  consider the bilevel program \eqref{BLPP} under the following assumption:
\begin{enumerate}
	\item[(H1)] $P \cap domS = P \cap domF$.
\end{enumerate}

Denote ${h_0}(p,x) := f(p,x) - \varphi (p)$. Introduce the sets
$$D := \{ (p,x) \mid  {{h_i}(p,x) \le 0,\ i \in I,\ g_j(p) \le 0,\ j \in J\} } ,\ C := \{ (p,x) \in D \mid  {{h_0}(p,x) \le 0\} }  .$$
and $
C(p):= \{ x \in \mathbb{R}^m \mid  (p,x)\in C \}.  
$
\begin{theorem}\label{theorem:4.1}
	Let $({p^0},{x^0})$ be a local solution of  problem \eqref{BLPP}. Suppose that the mapping $S$ is R-regular at $({p^0},{x^0})$ and the function $G$ is Lipschitz continuous on the set $D$ with Lipschitz constant ${l_0} > 0$. Then there exists a number ${\mu _0} > 0$ such that for any $\mu  > {\mu _0}$ the point $({p^0},{x^0})$is a local solution to the problem 
	\begin{equation}
	\begin{array}{ll}
	\text{minimize} &G(p,x) + \mu (f(p,x) - \varphi (p)) ,\\
	\text{s.t.} & (p,x) \in D.
	\end{array}
	\end{equation}
\end{theorem}
\begin{proof} It follows from (H1)  that $d(x,C(p)) = d((p,x),\{ p\}  \times C(p)) \ge d((p,x),C)$ for all $(p,x) \in D$.
	In virtue of Proposition 2.4.3 of \cite{MR709590}, for all $\alpha  > {l_0}$, the point $({p^0},{x^0})$ is a solution of the problem 
	\begin{equation}
	\begin{array}{ll}
	\text{minimize} & G(p,x) + \alpha d((p,x),C) ,\\
	\text{s.t.} & (p,x) \in D.
	\end{array}
	\end{equation}
	Then $d((p,x),C) \le d(x,C(p))$ for all $(p,x) \in D$ and, therefore, $G(p,x) + \alpha d(x,C(p)) \ge G({p^0},{x^0}) + \alpha d({x^0},C({p^0})) = G({p^0},{x^0})$.
	
	Since R-regularity for $S$ implies R-regularity for the set-valued $p\rightrightarrows C(p)$,  there exists a neighbourhood $V(p{^0},{x^0})$ such that $G(p,x) + \alpha M\max \{ 0,{h_0}(p,x)\}  \ge G({p^0},{x^0})$ for all  $(p,x) \in D \cap V(p{^0},{x^0})$. The last inequality is equivalent to the assertion of the theorem with ${\mu _0} = {l_0}M$.
\end{proof}

\bibliographystyle{plain}
\bibliography{references}
\end{document}